\documentclass[12pt,a4paper]{amsart}
\usepackage{amssymb}
\usepackage{calrsfs}
\usepackage{url}
\newcommand{\Q}{{\mathbb{Q}}}
\newcommand{\F}{{\mathbb{F}}}
\newcommand{\C}{{\mathbb{C}}}
\newcommand{\Z}{{\mathbb{Z}}}
\newcommand{\K}{{\mathbb{K}}}

\newcommand{\bB}{{\boldsymbol{B}}}
\newcommand{\bC}{{\boldsymbol{C}}}
\newcommand{\bG}{{\boldsymbol{G}}}
\newcommand{\bH}{{\boldsymbol{H}}}
\newcommand{\bL}{{\boldsymbol{L}}}
\newcommand{\bP}{{\boldsymbol{P}}}
\newcommand{\bT}{{\boldsymbol{T}}}
\newcommand{\bW}{{\boldsymbol{W}}}
\newcommand{\bZ}{{\boldsymbol{Z}}}
\newcommand{\fX}{{\mathfrak{X}}}
\newcommand{\fY}{{\mathfrak{Y}}}
\newcommand{\cE}{{\mathcal{E}}}
\newcommand{\cT}{{\mathcal{T}}}
\newcommand\Irr{\operatorname{Irr}}
\newcommand\Unip{\operatorname{Uch}}

\newtheorem{thm}{Theorem}[section]

\theoremstyle{definition}

\newtheorem{exmp}[thm]{Example}
\newtheorem{abs}[thm]{}

\theoremstyle{remark}
\newtheorem{rem}[thm]{Remark}

\renewcommand{\leq}{\leqslant}

\begin{document}
\title[Character tables of $F_4(2^f)$]{The character table of the finite
Chevalley group $F_4(q)$ for $q$ a power of~$2$}
\author{Meinolf Geck}
\address{Lehrstuhl f\"ur Algebra, FB Mathematik\\ Universit\"at Stuttgart\\ 
Pfaffenwaldring 57\\ 70569 Stuttgart, Germany}
\email{meinolf.geck@mathematik.uni-stuttgart.de}

\date{March 6, 2023}

\begin{abstract} Let $q$ be a prime power and $F_4(q)$ be the Chevalley 
group over a finite field with $q$ elements. Marcelo--Shinoda (1995) 
determined the values of the unipotent characters of $F_4(q)$ on all 
unipotent elements, extending earlier work by Kawanaka and Lusztig to
small characteristics. Assuming that $q$ is a power of~$2$, we explain
how to construct the complete character table of~$F_4(q)$. 
\end{abstract}

\maketitle

%%%%%%%%%%%%%%%%%%%%%%%%%%%%%%%%%%%%%%%%%%%%%%%%%%%%%%%%%%%%%%%%%%%%%%%%%%%
\section{Introduction} \label{sec0}

Let $p$ be a prime and $k=\overline{\F}_p$ be an algebraic closure of
the field with $p$ elements. Let~$\bG$ be a connected reductive algebraic
group over $k$ and assume that $\bG$ is defined over the finite subfield
$\F_q\subseteq k$, where $q$ is a power of~$p$. Let $F\colon \bG\rightarrow 
\bG$ be the corresponding Frobenius map. The finite group of fixed
points $\bG^F$ is called a ``finite group of Lie type''. We are concerned 
with the problem of computing the character table of $\bG^F$. The work of 
Lusztig \cite{LuB}, \cite{L7} has led to a general program for solving 
this problem. 

However, in concrete examples, there are still a certain number of 
technical~---~and sometimes quite intricate~---~issues to be resolved. 
In this paper, we show how this can be done for the groups $\bG^F=F_4(q)$, 
where $q$ is a power of~$2$. The conjugacy classes have been classified 
by Shinoda \cite{Shi}; the values of all unipotent characters on unipotent
elements were already determined by Marcelo--Shinoda \cite{MaSh}. A further 
crucial ingredient is the fact that the characteristic functions of the 
$F$-invariant cuspidal character sheaves of $\bG$ (for the definition, 
see \cite{L7} and the references there) are explicitly known as linear 
combinations of the irreducible characters of $\bG^F$. Building on earlier 
work of Shoji \cite{S2}, \cite{S3}, this has been achieved in 
\cite{MaSh}, \cite{pamq}.

In Section~2 we introduce basic notation and collect some general 
results from Lusztig's theory, where we use the books \cite{rDiMi}, 
\cite{gema} as our references. In Section~3 and 4 we focus on $\bG^F=
F_4(q)$. First we consider the unipotent characters of $\bG^F$. Then we
address some issues concerning the two-variable Green functions involved 
in Lusztig's cohomological induction functor which allows us, finally, 
to consider the non-unipotent characters.

The special feature of $\bG^F=F_4(q)$ as above is that the possible root 
systems of centralisers of semisimple elements are rather restricted. 
(See Remark~\ref{rem31} below.) There is a completely similar situation 
for $\bG$ of type~$E_6$ in characteristic~$2$, assuming that $\bG$ has a
connected centre and a simply connected derived subgroup. This, as well
as the case of groups of type $E_7$ in characteristic~$2$, will be
discussed in a sequel to this paper. The values of the unipotent 
characters on unipotent elements have been recently determined by Hetz 
\cite{Het3} for these groups.

I understand that Frank L\"ubeck has already prepared an electronic
``generic'' character table of $F_4(q)$, based on some assumptions 
concerning the values of the characteristic functions of 
certain $F$-invariant character sheaves on $\bG$. With the results of 
this paper, it should now be possible to verify those assumptions 
(or adjust them appropriately).

\smallskip
\begin{abs} {\bf Notation and conventions.} \label{sub00}
The set of (complex) irreducible characters of a finite group~$\Gamma$
is denoted by $\Irr(\Gamma)$. We work over a fixed subfield $\K\subseteq 
\C$, which is algebraic over $\Q$, invariant under complex conjugation 
and ``large enough'', that is, $\K$ contains sufficiently many roots of
unity and $\K$ is a splitting field for $\Gamma$ and all of its subgroups. 
In particular, $\chi(g)\in\K$ for all $\chi\in \Irr(\Gamma)$ and $g\in
\Gamma$. Let $\mbox{CF}(\Gamma)$ be the space of $\K$-valued class 
functions on $\Gamma$. There is a standard inner product $\langle \;,\;
\rangle_\Gamma$ on $\mbox{CF}(\Gamma)$ given by $\langle f,f'
\rangle_\Gamma:= |\Gamma|^{-1}\sum_{g \in \Gamma} f(g)\overline{f'(g)}$ 
for $f,f' \in \mbox{CF}(\Gamma)$, where $x\mapsto \overline{x}$ denotes 
the automorphism of $\K$ given by complex conjugation. We denote by
$\Z \Irr(\Gamma)\subseteq \mbox{CF}(\Gamma)$ the subset consisting of
all integral linear combinations of $\Irr(\Gamma)$. Finally, if $C\subseteq 
\Gamma$ is any (non-empty) subset that is a union of conjugacy classes
of $\Gamma$, then we denote by $\varepsilon_C \in\mbox{CF}(\Gamma)$ 
the (normalised) indicator function of $C$, that is, we have 
\[\varepsilon_C(g)=\left\{\begin{array}{cl} |\Gamma|/|C| & \qquad
\mbox{if $g\in C$},\\ 0 & \qquad \mbox{otherwise}.\end{array}\right.\]
Note that, if $C$ is a single conjugacy class of $\Gamma$ and $g\in C$, 
then $f(g)=\langle f,\varepsilon_C\rangle_{\Gamma}$ for any $f\in 
\mbox{CF}(\Gamma)$. Thus, the problem of computing the values of 
$\rho\in \Irr(\bG^F)$ is equivalent to working out the inner products 
of $\rho$ with the indicator functions of the various conjugacy classes
of $\Gamma$. 
\end{abs}

%If $\alpha\colon 
%\Gamma\rightarrow\Gamma$ is a group automorphism, we say that $g_1,g_2
%\in\Gamma$ are $\alpha$-conjugate if there exists some $g\in\Gamma$ such
%that $g_2=g^{-1}g_1\alpha(g)$.

%%%%%%%%%%%%%%%%%%%%%%%%%%%%%%%%%%%%%%%%%%%%%%%%%%%%%%%%%%%%%%%%%%%%%%%%%%%
\section{Lusztig induction and uniform functions} \label{secrlg}

Let $\bG,F$ be as in the introduction. Given an $F$-stable maximal torus
$\bT$ of~$\bG$ and $\theta\in \Irr(\bT^F)$, we have a generalised 
character $R_{\bT,\theta}^\bG \in \Z \Irr(\bG^F)$ as introduced by 
Deligne and Lusztig \cite{DeLu} (see also \cite[\S 2.2]{gema}). We shall 
also need the following generalisation of $R_{\bT,\theta}^\bG$. 

\begin{abs} \label{abs11} An $F$-stable closed subgroup $\bL\subseteq \bG$ 
is called a ``regular subgroup'' if $\bL$ is a Levi complement in some (not 
necessarily $F$-stable) parabolic subgroup $\bP\subseteq \bG$. Given 
such a pair $(\bL,\bP)$ we obtain an operator
\[R_{\bL\subseteq \bP}^{\bG} \colon \Z\Irr(\bL^F)\rightarrow \Z\Irr(\bG^F)
\qquad\mbox{(``Lusztig induction''; see \cite[\S 9.1]{rDiMi})}.\]
Denoting by $\bG_{\text{uni}}^F$  and $\bL_{\text{uni}}^F$ the sets of 
unipotent elements of $\bG^F$ and $\bL^F$, respectively, there is a 
corresponding two-variable Green function 
\[ Q_{\bL\subseteq \bP}^{\bG}\colon \bG_{\text{uni}}^F\times 
\bL_{\text{uni}}^F\rightarrow\Q\qquad\mbox{(see 
\cite[\S 10.1]{rDiMi})}.\]
If $\bL=\bT$ is an $F$-stable maximal torus of $\bG$ (and $\bB\subseteq 
\bG$ is a Borel subgroup containing~$\bT$), then $\bT_{\text{uni}}^F=\{1\}$ 
and $Q_{\bT}^\bG \colon \bG_{\text{uni}}^F \rightarrow \Q$, $u\mapsto 
Q_{\bT\subseteq \bB}^{\bG}(u,1)$, is the ``usual'' Green function 
originally introduced in \cite{DeLu}, that is, we have
$Q_\bT^\bG(u)=R_{\bT,1}^{\bG}(u)$ for all $u\in \bG_{\text{uni}}^F$.
\end{abs}

\begin{abs} \label{abs12} Let $\bL\subseteq \bP$ be as above and 
$\psi\in \Irr(\bL^F)$. There is a character formula which expresses 
the values of $R_{\bL\subseteq \bP}^\bG(\psi)$ in terms of the values 
of~$\psi$ and the two-variable Green functions for $\bG$ and for groups 
of the form $C_{\bG}^\circ(s)$ where $s\in \bG^F$ is semisimple; 
see \cite[Prop.~10.1.2]{rDiMi}, \cite[Prop.~6.2]{L3} for the precise 
formulation. For later reference we only state here the following
special case:
\begin{equation*}
R_{\bL\subseteq \bP}^{\bG}(\psi)(u)=\sum_{v \in \bL_{\text{uni}}^F}
Q_{\bL\subseteq \bP}^{\bG}(u,v)\psi(v) \qquad\mbox{for all
$u\in \bG_{\text{uni}}^F$}.\tag{a}
\end{equation*}
We also state the following useful formula. Let $g\in \bG^F$ and consider 
the Jordan decomposition of $g$, that is, we write $g=su=us$ where 
$s\in \bG^F$ is semisimple and $u\in \bG^F$ is unipotent. If 
$C_{\bG}^\circ(s)\subseteq \bL$, then
\begin{equation*}
\rho(g)=\sum_{\psi \in \Irr(\bL^F)} \big\langle R_{\bL\subseteq 
\bP}^{\bG}(\psi),\rho\big\rangle_{\bG^F}\psi(g)\qquad\mbox{for all
$\rho\in \Irr(\bG^F)$}.\tag{b}
\end{equation*}
This appeared in K. D. Schewe's dissertation (Bonner Mathematische 
Schrif\-ten, vol.~165, 1985); see the remark following 
\cite[Cor.~3.3.13]{gema} for a proof.
\end{abs}

\begin{abs} \label{abs13} Let us denote by $\fX(\bG,F)$ the set of all 
pairs $(\bT,\theta)$ where $\bT\subseteq \bG$ is an $F$-stable maximal 
torus and $\theta\in \Irr(\bT^F)$. Following \cite[p.~16]{cbms}, a class
function $f\in \mbox{CF}(\bG^F)$ is called ``uniform'' if $f$ can be
written as a $\K$-linear combination of the generalised characters 
$R_{\bT,\theta}^{\bG}$ for various pairs $(\bT, \theta)\in \fX(\bG,F)$. 
If $f$ is uniform, then we have (see \cite[Prop.~10.2.4]{rDiMi}):
\[ f=|\bG^F|^{-1}\sum_{(\bT,\theta)\in \fX(\bG,F)} |\bT^F| \langle 
f,R_{\bT,\theta}^\bG\rangle_{\bG^F}R_{\bT,\theta}^\bG.\]
For example, if $C$ is a conjugacy class of semisimple elements of 
$\bG^F$, then the indicator function $\varepsilon_C$ (as in (\ref{sub00}))
is uniform; see \cite[Cor.~10.3.4]{rDiMi}.
\end{abs}

\begin{thm} \label{unifc} Let $\bC$ be an arbitrary $F$-stable conjugacy 
class of $\bG$. Then the indicator function $\varepsilon_{\bC^F}$ of the 
set $\bC^F$ is a uniform function. \\ (Note that, in general, $\bC^F$ is 
a union of conjugacy classes of~$\bG^F$.)
\end{thm}

\begin{proof} See the appendix of \cite{mylaus}; this was conjectured 
by Lusztig \cite[2.16]{cbms}. See also \cite[Cor.~13.3.5]{rDiMi} and 
\cite[Theorem~2.7.11]{gema}. 
\end{proof}

\begin{exmp} \label{unifc1} Let $g\in \bG^F$ and assume that $C_\bG(g)$
is connected. Let $\bC$ be the $\bG$-conjugacy class of $g$. Since $C_\bG
(g)$ is connected, $C:=\bC^F$ is a single conjugacy class of $\bG^F$; see 
\cite[Example~1.4.10]{gema}. Now $\varepsilon_C$ is uniform by
Theorem~\ref{unifc}. Let $\rho\in \Irr(\bG^F)$. Recall from (\ref{sub00}) 
that $\rho(g)= \langle \rho,\varepsilon_C\rangle_{\bG^F}$ and 
$\langle \varepsilon_C,R_{\bT,\theta}^\bG\rangle_{\bG^F}=R_{\bT,
\theta^{-1}}^\bG(g)$ for any $(\bT,\theta)\in \fX(\bG,F)$. Hence, using 
(\ref{abs13}), we obtain the formula:
\[ \rho(g)=|\bG^F|^{-1}\sum_{(\bT,\theta)\in \fX(\bG,F)} |\bT^F|\,
\langle R_{\bT,\theta}^\bG,\rho\rangle_{\bG^F}\,R_{\bT,\theta^{-1}}^\bG(g).\]
This shows that the value $\rho(g)$ is determined by the multiplicities
$\langle R_{\bT,\theta}^\bG,\rho \rangle_{\bG^F}$ and the values
$R_{\bT,\theta}^\bG(g)$, where $(\bT,\theta)$ runs over all pairs in 
$\fX(\bG,F)$.
\end{exmp}

\begin{abs} \label{abs21} We say that $\rho\in \Irr(\bG^F)$ is ``unipotent'' 
if $\langle R_{\bT,1}^{\bG},\rho\rangle_{\bG^F} \neq 0$ for some 
$F$-stable maximal torus $\bT\subseteq \bG$. We denote by $\Unip(\bG^F)$ 
the set of unipotent characters of $\bG^F$. As shown in Lusztig's book 
\cite{LuB}, these characters play a special role in the character theory 
of $\bG^F$; many questions about arbitrary characters of $\bG^F$ can be 
reduced to unipotent characters. 
\end{abs}

\section{The unipotent characters for $F_4$ in characteristic~$2$} 
\label{typ4}

%We have $\bG=\langle x_\alpha(\xi)\mid \alpha\in \Phi,\xi\in k\rangle$ where 
%$\Phi$ is the root system of $G$ with respect to a fixed $F$-stable maximal 
%torus~$\bT_0 \subseteq \bG$ that is contained in an $F$-stable Borel subgroup 
%$\bB_0\subseteq \bG$. We may assume that $F(t)=t^q$ for all $t\in \bT_0$ 
%and $F(x_\alpha(\xi))=x_\alpha(\xi^q)$ for all $\alpha\in \Phi$ and $\xi
%\in k$. Let $\Delta=\{\alpha_1,\alpha_2,\alpha_3,\alpha_4\}$ be the set 
%of simple roots with respect to~$\bB_0$, where the labelling is according 
%to the following Dynkin diagram.
%\begin{center}
%\begin{picture}(210,22)
%\put( 61,15){$\alpha_1$}
%\put( 91,15){$\alpha_2$}
%\put(121,15){$\alpha_3$}
%\put(151,15){$\alpha_4$}
%\put( 65,07){\circle*{6}}
%\put( 95,07){\circle*{6}}
%\put(125,07){\circle*{6}}
%\put(155,07){\circle*{6}}
%\put( 65,07){\line(1,0){30}}
%\put( 95,05){\line(1,0){30}}
%\put( 95,09){\line(1,0){30}}
%\put(125,07){\line(1,0){30}}
%\put(105,3.8){$>$}
%\end{picture}
%\end{center}
We assume from now on that $p=2$ and $\bG$ is simple of 
type $F_4$. Let $F\colon \bG \rightarrow \bG$ be a Frobenius map such that
$\bG^F=F_4(q)$ where~$q$ is a power of~$2$. 
Let $\fY(\bG,s)$ be the set of all pairs $(\bT,s)$ where $\bT\subseteq
\bG$ is an $F$-stable maximal torus and $s\in \bT^F$. There are natural
actions of $\bG^F$ on $\fX(\bG,F)$ and on $\fY(\bG,F)$; see \cite[2.3.20 
and 2.5.12]{gema}. Since $\bG\cong \bG^*$ is ``self-dual'' (in the sense 
of \cite[Def.~1.5.17]{gema}), there is a bijective correspondence 
\[\fX(\bG,F) \;\, \mbox{mod $\bG^F$} \;\;\leftrightarrow \;\; 
\fY(\bG,F)\;\, \mbox{mod $\bG^F$} \qquad \mbox{(see 
\cite[Cor.~2.5.14]{gema})}.\]
If $(\bT,\theta)\leftrightarrow (\bT,s)$ correspond in this way, we write 
$R_{\bT,s}^\bG:= R_{\bT,\theta}^\bG$ (see \cite[Def.~2.5.17]{gema}).
In order to compute the characters of $\bG^F$, we shall assume that 
the following information is known and available in the form of tables:
\begin{itemize}
\item[{\bf (A1)}] Parametrisations of $\fY(\bG,F)$ and of all the 
conjugacy classes of $\bG^F$.
\item[{\bf (A2)}] The multiplicities $\langle R_{\bT,s}^{\bG},\rho\rangle$
for all $\rho\in \Irr(\bG^F)$ and $(\bT,s)\in \fY(\bG,F)$.
\item[{\bf (A3)}] The values $R_{\bT,s}^\bG(g)$ for all $g\in \bG^F$ and 
all $(\bT,s) \in \fY(\bG,F)$.
\item[{\bf (A4)}] For every regular $\bL\subsetneqq \bG$, the 
values $\psi(u)$ for $\psi\in \Irr(\bL^F)$, $u\in \bL_{\text{uni}}^F$.
\end{itemize}
%An irreducible character $\rho\in \Irr(\bG^F)$ is called ``unipotent'' 
%if $\langle R_{\bT,1}^{\bG},\rho\rangle_{\bG^F} \neq 0$ for some 
%$F$-stable maximal torus $\bT\subseteq \bG$. We denote by $\Unip(\bG^F)$ 
%the set of unipotent characters of $\bG^F$; in our group $\bG^F$, there
%are precisely $37$ unipotent characters (see \cite[p.~371/372]{LuB}). 
%Then we also assume that the following information is known:
%\begin{itemize}

\begin{rem} \label{rem31} The conjugacy classes of $\bG^F$ are determined
by Shinoda \cite{Shi}. The tables in \cite{Shi} provide the required 
classifications and parametrisations in~{\bf (A1)}. Since the center 
of~$\bG$ is trivial, the information in {\bf (A2)} is available via
Lusztig's ``Main Theorem~4.23'' in \cite{LuB}; see also \cite[\S 2.4, 
\S 4.2]{gema}. In order to obtain {\bf (A3)}, one uses the character 
formula in \cite[\S 4]{DeLu} (see also \cite[Theorem~2.2.16]{gema}) for 
the evaluation of $R_{\bT,s}^\bG(g)$. This involves the Green functions 
for $\bG$ and for groups of the form $\bH_s=C_\bG(s)$ where $s\in\bG^F$ 
is semisimple; note that, for our $\bG$, the centraliser of any
semisimple element is connected. By inspection of \cite[Table~III]{Shi}, 
we see that $\bH_s$ is either a maximal torus, or a regular subgroup (with 
a root system of type $F_4$, $B_3$, $C_3$, $A_1\times A_2$, $B_2$, $A_2$, 
$A_1 \times A_1$ or $A_1$) or~$\bH_s$ has a root system of type $A_2\times 
A_2$. The Green functions for $\bG^F$ itself have been determined by Malle 
\cite{Mal1}; for the other cases see L\"ubeck \cite[Tabelle~16]{lphd}. The
further technical issues in the evaluation of $R_{\bT,\theta}^\bG(su)$ are
discussed in \cite[\S 3]{pamq} and \cite[\S 2]{lphd} (for example, one has 
to deal with a sum over all $x\in \bG^F$ such that $x^{-1}sx\in \bT$); in 
\cite[\S 6]{lphd}, this is explained in detail for the groups $\bG^F=
\mbox{CSp}_6(q)$. Finally, the required values in {\bf (A4)} can be 
extracted from Enomoto \cite{En} (type~$B_2$), Looker \cite{Loo}, 
L\"ubeck \cite[Tabelle~27]{lphd} (type $B_3,C_3$) and Steinberg 
\cite{St} (type $A_1,A_2$).
\end{rem}

%(An independent verification is provided by \cite[Remark~5.6]{ekay}.) 
%For the remaining groups~$\bH_s^F$, the Green functions are already listed 

% Write $g=su$ where $s\in \bG^F$ is semisimple and $u\in \bG^F$ is 
%unipotent. Then $\bH_s:=C_{\bG}(s)$ is connected and we have
%\[ R_{\bT,\theta_s}^{\bG}(g)=\frac{1}{|\bH_s^F|} \sum_{x\in\bG^F\,:
%\,x^{-1}sx\in\bT} Q_{x\bT x^{-1}}^{\bH_s}(u) \,\theta_s(x^{-1}sx),\]
%where $\theta_s\in \Irr(\bT^F)$ corresponds to $s\in \bT$ via the
%above bijection between $\fX(\bG,F) \, \mbox{mod $\bG^F$}$ and 
%$\fY(\bG,F)\,\mbox{mod $\bG^F$}$. The evaluation of the terms
%$\theta_s(x^{-1}sx)$ is handled by the simplified formula in 
%\cite[Lemma~2.2.23]{gema}. (See also \cite[\S 3]{pamq}; in L\"ubeck 
%\cite[\S 6]{lphd}, this technical issue is explained in detail for the 
%groups $\bG^F=\mbox{CSp}_6(q)$.) Thus, it remains to know the Green 
%functions for the groups $\bH_s^F$ where $s\in \bG^F$ is semisimple.

%Values of $R_\phi$ on $g=su$ where $C_G(s)$ of type $A_2\times A_2$ and
%$u$ regular unipotent:
%
%\[\begin{array}{l|cccccc} \hline &R_{\phi_{1,0}} &R_{\phi_{9,2}} &
%R_{\phi_{6,6}'}& R_{\phi_{8,3}''}&R_{\phi_{8,3}'}\\\hline 
%\mbox{$C_G(s)$ split} & 1 & 1 & 1 & 1 & 1 \\ \mbox{$C_G(s)$ non-split} 
%& 1 & 1 & 1 & -1 & -1 \\ \hline\end{array}\]
%
% character table of $F_4(\F_2)$;
%see the {\sf ATLAS} \cite[p.~167]{atl}.
%
%Lusztig induction \cite[\S 3.3]{gema}.

Representatives for the $\bG^F$-conjugacy classes of semisimple elements
are denoted by $h_0,h_1,\ldots,h_{76}$ in \cite[Table~II]{Shi}, where 
$h_0=1$; note that some of the~$h_i$ only occur according to whether 
$3\mid q-1$ or $3\mid q+1$, or when $q$ is sufficiently large. We now go 
through the list of these elements and explain how to determine the 
values of any unipotent character $\rho\in \Unip(\bG^F)$ on elements 
of the form $h_iu$ where $u \in C_\bG(h_i)^F$ is unipotent. 

In our group $\bG$, there are $37$ unipotent characters, where we use 
the notation in Lusztig's book \cite[p.~371/372]{LuB}). 
%Furthermore,
%by \cite[Theorem~2.1]{Shi}, the group $\bG^F$ has $35$ conjugacy classes 
%of unipotent elements.

\begin{abs} \label{abs32} If $s=h_0=1$, then the values $\rho(u)$ for 
$\rho\in \Unip(\bG^F)$ and $u\in \bG_{\text{uni}}^F$ have been explicitly
determined by Marcelo--Shinoda; see \cite[Table~6.A]{MaSh}. This relies on 
the Green functions of~$\bG^F$ (available from \cite{Mal1}) and also on the 
knowledge of the ``generalised Green functions'' arising from Lusztig's
theory of character sheaves. An algorithm for the computation of those 
functions is described in \cite[\S 24]{L2e}; it involves the delicate 
matter of normalising certain ``$Y_\iota$-functions'' (defined in 
\cite[(24.2.3)]{L2e}). Marcelo--Shinoda \cite{MaSh} do not explain in 
detail how they found those normalisations. But using the argument of 
Hetz \cite[\S 4.1.4]{Het3} (where the analogous problem is solved for 
groups of type $E_6$ in characteristic~$2$), one obtains an independent
verification that the values in \cite[Table~5]{MaSh} are correct.
\end{abs}

\begin{abs} \label{abs34} Let $s=h_3$ (if $3\mid q-1$) or $s=h_{15}$ 
(if $3\mid q+1$). Then $\bH_s=C_\bG(s)$ has a root system of type
$A_2 \times A_2$. Let $u\in\bH_s^F$ be unipotent and $\bC$ be the 
$\bG$-conjugacy class of $su$. 

(a) Assume first that $u$ is not regular unipotent. By inspection
of \cite[Table~IV]{Shi}, we see that $C_\bG(su)$ is connected. So we
can apply Example~\ref{unifc1}, together with {\bf (A2)}, {\bf (A3)}, 
to determine $\rho(su)$ even for all $\rho \in \Irr(\bG^F)$.

(b) Now assume that $u$ is regular unipotent. We recall some facts 
from \cite[\S 7.6]{pamq}. (Note that, in \cite[\S 7.6]{pamq} it is 
assumed that $p\neq 2,3$ but the discussion works verbatim also for 
$p=2$.) The set $\bC^F$ splits into $3$ classes in $\bG^F$, which we
simply denote by $C_1,C_2,C_3$. We can choose the notation such that 
$C_1=C_1^{-1}$ and $C_2^{-1}=C_3$. Explicit representatives are 
described in \cite[Table~IV]{Shi}; we have $|C_\bG(g_i)^F|=3q^4$ for 
$g_i\in C_i$ and $i=1,2,3$. Let $\chi_0:=\varepsilon_{\bC^F}$ be the 
indicator function on the set $\bC^F$ (as in (\ref{sub00})). Let $1\neq 
\theta \in \K$ be a fixed third root of unity. Then we consider the 
following linear combinations of unipotent characters of $\bG^F$: 
\begin{align*}
\chi_1 & :=\textstyle\frac{1}{3}q^2\bigl([12_1]+
F_4^{\rm I\!I}[1]-[6_1]-[6_2]+2F_4[\theta]-F_4[\theta^2]\bigr),\\
\chi_2 & :=\textstyle\frac{1}{3}q^2\bigl([12_1]+
F_4^{\rm I\!I}[1]-[6_1]-[6_2]-F_4[\theta]+2F_4[\theta^2]\bigr).
\end{align*}
As discussed in \cite[\S 7.6]{pamq}, the class functions $\chi_1,\chi_2$ 
are (scalar multiples of) characteristic functions of $F$-invariant 
cuspidal character sheaves on $\bG$; furthermore, the values of
$\chi_0,\chi_1,\chi_2$ are given as follows:
\[ \renewcommand{\arraycolsep}{10pt} \begin{array}{ccccc} \hline 
& C_1 & C_2 & C_3 & g \in \bG^F\setminus \bC_s^F\\\hline 
\chi_0 & q^4 & q^4 & q^4 & 0 \\ \chi_1 & q^4 & q^4\theta & q^4\theta^2 
& 0 \\ \chi_2 & q^4 & q^4\theta^2 & q^4 \theta & 0 \\\hline\end{array}\]
Hence, $\;\;\varepsilon_{C_1}=\chi_0+\chi_1+\chi_2$, $\;
\varepsilon_{C_2}=\chi_0+\theta^2\chi_1+\theta\chi_2$, 
$\;\varepsilon_{C_3}=\chi_0+ \theta\chi_1+\theta^2\chi_2$.

Now let $\rho\in \Irr(\bG^F)$ be arbitrary and $g_i\in C_i$ for $i=1,2,3$.
Since $\chi_0$ is uniform by Theorem~\ref{unifc}, we can determine
$\langle \rho,\chi_0\rangle_{\bG^F}$ using {\bf (A2)}, {\bf (A3)} and 
the formula in (\ref{abs13}). The inner products of $\rho$ with
$\chi_1,\chi_2$ are known by the definition of $\chi_1,\chi_2$. Hence, 
we can explicitly work out $\rho(g_i)=\langle \rho,\varepsilon_{C_i}
\rangle_{\bG^F}$.
\end{abs}

%\begin{abs} \label{abs33} Let $s=h_i$ where $i\not\in \{0,1,2,3,6,13,
%14,15,18,26,31,53\}$. In these cases, $\bH_s=C_{\bG}(s)$ either is a
%maximal torus, or a proper regular subgroup of type $A_1\times A_2$, 
%$A_2$, $A_1\times A_1$ or $A_1$. Let $u\in \bH_s^F$ be unipotent and 
%$\bC$ be the $\bG$-conjugacy class of $su$. By inspection of
%\cite[Table~IV]{Shi}, we see that $C_\bG(g)$ is connected. Hence, 
%$\bC^F$ is a single $\bG^F$-conjugacy class and we can apply 
%Corollary~\ref{unifc1} to determine $\rho(su)$ for $\rho \in \Irr(\bG^F)$.
%\end{abs}

\begin{abs} \label{abs33} Let $s=h_i$ where $i\not\in \{0,3,15\}$.
In these cases, $\bL=C_{\bG}(s)$ either is a maximal torus, or a proper 
regular subgroup with a root system of type $B_3$, $C_3$, $A_1\times A_2$,
$B_2$, $A_2$, $A_1\times A_1$ or~$A_1$. Let $u\in \bL^F$ be unipotent
and $\bC$ be the $\bG$-conjugacy class of $su$. Let $\rho\in \Unip(\bG^F)$.
In order to compute $\rho(su)$, we use Schewe's formula in (\ref{abs12}).
First note that, if $\psi\in \Irr(\bL^F)$ is such that $\big\langle 
R_{\bL\subseteq \bP}^{\bG} (\psi),\rho\big\rangle_{\bG^F}\neq 0$, then
we must have $\psi\in \Unip(\bL^F)$; see \cite[Prop.~3.3.21]{gema}. 
Furthermore, since $s$ is in the centre of $\bL^F$, we have $\psi(su)=
\psi(u)$. (This is a general property of unipotent characters; see
\cite[Prop.~2.2.20]{gema}.) Hence, Schewe's formula reads:
\[ \rho(su)=\sum_{\psi\in \Unip(\bL^F)} \big\langle R_{\bL\subseteq 
\bP}^{\bG}(\psi),\rho\big\rangle_{\bG^F}\psi(u).\]
By {\bf (A4)}, the values $\psi(u)$ for $\psi\in \Unip(\bL^F)$ and 
$u\in \bL_{\text{uni}}^F$ are explicitly known. The multiplicities 
$\langle R_{\bL\subseteq \bP}^{\bG}(\psi),\rho \rangle_{\bG^F}$ (for 
$\rho\in \Unip(\bG^F)$ and $\psi\in \Unip(\bL^F)$) can also be determined 
explicitly; see \cite[\S 4.6]{gema}, especially \cite[Prop.~4.6.18]{gema}. 
In Michel's version of {\sf CHEVIE} \cite{gap3jm}, this is available 
through the function {\tt LusztigInductionTable}. Let us illustrate 
this with an example.
\end{abs}

\begin{table}[htbp] \caption{Unipotent characters for type $B_2$ in
characteristic~$2$} \label{uniSp4}
\begin{center}
{\small $\renewcommand{\arraystretch}{1.1}
\begin{array}{c@{\hspace{7pt}}c@{\hspace{10pt}}c@{\hspace{10pt}}
c@{\hspace{10pt}}c@{\hspace{10pt}}c@{\hspace{10pt}}c@{\hspace{10pt}}} \hline
& A_1 & A_2 & A_{31} & A_{32} & A_{41} & A_{42}\\\hline
%& O_{(1111)} & O_{(211)} & O_{(22)}^* & O_{(22)} & O_{(4)} & O_{(4)}' 
%\\ \hline 
%& 1 & x_{2a+b}(1) & x_{a+b}(1) & x_{a+b}(1)x_{2a+b}(1) & x_a(1)x_b(1) 
%& x_a(1)x_b(1)x_{2a+b}(\xi) \\ 
|C_{\bG}(u)^F|: & q^4(q^2{-}1)(q^4{-}1) & q^4(q^2{-}1) & q^4(q^2{-}1) & 
q^4 & 2q^2 & 2q^2 \\\hline
\psi_0 
%(1_W)
& 1 & 1 & 1 & 1 & 1 & 1 \\
\psi_9 
%(\phi_1)
&\frac{1}{2}q(q+1)^2 & \frac{1}{2}q(q+1) & \frac{1}{2}q(q+1)
& \frac{q}{2} & \frac{q}{2} &-\frac{q}{2} \\ 
\psi_{10} 
%(\text{cusp})
& \frac{1}{2}q(q-1)^2 & -\frac{1}{2}q(q-1)& -\frac{1}{2}q(q-1)
& \frac{q}{2} & \frac{q}{2} & -\frac{q}{2}\\ 
\psi_{11} 
%(\epsilon')
&\frac{1}{2}q(q^2+1) & -\frac{1}{2}q(q-1) & \frac{1}{2}q(q+1)
& \frac{q}{2} & -\frac{q}{2} &\frac{q}{2} \\ 
\psi_{12} 
%(\epsilon'')
&\frac{1}{2}q(q^2+1) & \frac{1}{2}q(q+1) & 
-\frac{1}{2}q(q-1) & \frac{q}{2} & -\frac{q}{2} &\frac{q}{2} \\ 
\psi_{13} 
%(\epsilon)
& q^4 & . & . & . & . & . \\ \hline
\multicolumn{7}{c}{\text{(See Enomoto \cite{En}; notation as in 
\cite[Examples~3.3.30 and~2.7.22]{gema}.)}}
\end{array}$} \end{center}
\end{table}

\begin{exmp} \label{exp33}  Let $\rho=F_4^{\rm I\!I}[1]
\in \Unip(\bG^F)$ (a cuspidal unipotent character). Let $s=h_{53}$; then 
$\bL=C_\bG(s)$ is a regular subgroup of type $B_2$, where $|\bL^F|=
q^4(q^2+1)(q^2-1)(q^4-1)$; see \cite[Table~III]{Shi}. We would like to 
determine the values $\rho(h_{53}u)$ where $u\in \bL^F$ is unipotent. 
The values of the unipotent characters of $\bL^F$ on unipotent elements 
are given by Table~\ref{uniSp4}. Using Michel's {\tt LusztigInductionTable}, 
we find that 
\[ \langle R_{\bL\subseteq \bP}^{\bG}(\psi_{10}),\rho\rangle_{\bG^F}=1
\qquad \mbox{and}\qquad \langle R_{\bL\subseteq \bP}^{\bG}(\psi_i),
\rho\rangle_{\bG^F}=0\quad \mbox{for $i\neq 10$}.\]
Hence, by Schewe's formula, we have $\rho(h_{53}u)=\psi_{10}(u)$.~---~A 
completely analogous procedure works for any $s=h_i$ as in (\ref{abs33}).
\end{exmp}

\section{Non-unipotent characters for $F_4$ in characteristic $2$} 
\label{twogr}

%For $s=h_0$ (the identity element), 
%we could just cite Marcelo--Shinoda \cite{MaSh} for the values of 
%unipotent characters on unipotent elements. For non-unipotent characters, 
%some extra work is required where we use the two-variable Green functions
%from Section~\ref{secrlg}.

We keep the notation of the previous section, where $\bG$ is simple
of type~$F_4$ in characteristic $2$. We now explain how to determine the
values of the non-unipotent characters of~$\bG^F$. First we recall some
facts from Lusztig's classification of $\Irr(\bG^F)$. Let $s\in \bG^F$
be semisimple. Then we define $\cE(\bG^F,s)$ to be the set of all $\rho
\in \Irr(\bG^F)$ such that $\langle R_{\bT,s}^\bG, \rho\rangle\neq 0$ for
some $F$-stable maximal torus $\bT\subseteq \bG$ with $s\in \bT$. It is 
known that every $\rho\in \Irr(\bG^F)$ belongs to $\cE(\bG^F,s)$ for 
some~$s$; furthermore, $\cE(\bG^F,s)$ only depends on the $\bG^F$-conjugacy
class of~$s$. If $s,s'\in \bG^F$ are such that $\cE(\bG^F,s)\cap \cE(\bG^F,
s')\neq \varnothing$, then $s,s'$ are $\bG^F$-conjugate. (For all this
see, for example, \cite[\S 2.6]{gema}; also recall that $\bG\cong \bG^*$.) 
Finally, by the ``Main Theorem 4.23'' of \cite{LuB}, there is a 
bijection $\cE(\bG^F,s)\leftrightarrow \Unip(\bH_s^F)$, where $\bH_s=
C_\bG(s)$; this is called the ``Jordan decomposition'' of characters.
We now proceed in 4 steps, where we determine the following information:

\smallskip
{\bf Step 1}: The values of all the two-variable Green functions
$Q_{\bL\subseteq \bP}^\bG$.

{\bf Step 2}: The values $\rho(u)$ for all $\rho\in \Irr(\bG^F)$
and $u \in \bG_{\text{uni}}^F$.

{\bf Step 3}: The decomposition of $R_{\bL\subseteq \bP}^\bG(\psi)$
for any $\psi \in \Irr(\bL^F)$.

{\bf Step 4}: The values $\rho(g)$ for any $\rho\in
\Irr(\bG^F)$ and any $g\in \bG^F$.

%For unipotent elements, that formula takes the following simple form:
%\begin{equation*}
%R_{\bL\subseteq \bP}^{\bG}(\psi)(u)=\sum_{v \in \bL_{\text{uni}}^F}
%Q_{\bL\subseteq \bP}^{\bG}(u,v^{-1})\psi(v) \qquad\mbox{for $\psi\in 
%\Irr(\bL^F)$ and $u\in \bG_{\text{uni}}^F$}.\tag{$\spadesuit$}
%\end{equation*}
%(See \cite[8.12(a)]{L5}.)
\begin{abs} \label{step1} We show how Step 1 can be resolved. Assume 
that $\bL\subsetneqq \bG$ and let $\Unip(\bL^F)=\{\psi_1,\ldots,\psi_n\}$. 
The information in {\bf (A4)} (see Section~\ref{typ4}) shows, in particular,
that $n$ is also the number of conjugacy classes of unipotent elements of 
$\bL^F$. Let $v_1,\ldots,v_n$ be representatives of these classes. 
Then, again using {\bf (A4)}, we can also check that the matrix 
$(\psi_i(v_j))_{1\leq i,j\leq n}$ is invertible. (For an example, see
Table~\ref{uniSp4}.) Let $u_1,\ldots,u_N$ be representatives of the 
conjugacy classes of unipotent elements of $\bG^F$; we have $N=35$ by 
\cite[Theorem~2.1]{Shi}. Then we write the character formula 
(\ref{abs12})(a) as a system of equations:
\begin{equation*}
R_{\bL\subseteq \bP}^{\bG}(\psi_i)(u_k)=\sum_{j=1}^n c_j \, Q_{\bL
\subseteq \bP}^{\bG}(u_k,v_j)\psi_i(v_j)\;\;\, \mbox{for
$1\leq i\leq n, 1\leq k \leq N$},\tag{$\spadesuit$}
\end{equation*}
where $c_j:=[\bL^F:C_{\bL}(v_j)^F]$ for all~$j$.
%We can invert these equations and obtain coefficients $s_i\in \K$ such that
%\[ Q_{\bL\subseteq \bP}^\bG(u_k,v_j^{-1})=\sum_{i=1}^n s_i
%R_{\bL\subseteq \bP}^\bG(\psi_i)(u_k) \qquad \mbox{for $1\leq j \leq n,
%1 \leq k\leq N$}.\]
On the other hand, as explained in (\ref{abs33}), we can determine the
multiplicities $m(\psi_i,\rho):=\langle R_{\bL\subseteq\bP}^\bG(\psi_i),
\rho \rangle_{\bG^F}$ for any $\rho \in \Unip(\bG^F)$. Hence, we obtain 
equations
\[ R_{\bL\subseteq \bP}^\bG(\psi_i)(u_k)=\sum_{\rho\in \Unip(\bG^F)}
m(\psi_i,\rho)\rho(u_k) \qquad \mbox{for $1\leq i \leq n, 1\leq k\leq N$}.\]
Consequently, since the values $\rho(u_k)$ for $\rho\in \Unip(\bG^F)$ are 
known by (\ref{abs32}), the values $R_{\bL \subseteq \bP}^{\bG}(\psi_i)
(u_k)$ can be computed explicitly. We can now invert ($\spadesuit$) and 
obtain all the values $Q_{\bL\subseteq \bP}^\bG(u_k,v_j)$ for 
$1\leq j \leq n$, $1\leq k\leq N$. (A similar argument appears in 
Malle--Rotilio \cite[\S 2.2]{MR}.)
\end{abs}

\begin{abs} \label{step2} We show how Step 2 can be resolved. As in the 
previous section, we consider the list of semisimple elements $h_0,h_1,
\ldots,h_{76}\in\bG^F$. Let $\rho \in \Irr(\bG^F)$. There is some $s\in 
\{h_0,h_1,\ldots,h_{76}\}$ such that $\rho\in \cE(\bG^F,s)$. If $s=h_0$ 
(the identity element), then $\rho$ is unipotent and the required values 
are known by (\ref{abs32}). Now assume that $s\in \{h_3,h_{15}\}$ where 
$C_\bG(s)$ has a root system of type $A_2\times A_2$. Then, by the 
discussion in \cite[Lemma~2.4.18]{gema} (which is drawn from Lusztig's 
book \cite{LuB}), we know that $\rho$ is a uniform class function. 
(The group $\bW_{\lambda,n}$ occurring in that discussion is isomorphic 
to the Weyl group of $C_\bG(s)$; see \cite[(2.5.10)]{gema} and note again 
that $\bG\cong \bG^*$.) Hence, the values $\rho(u)$ for $u\in 
\bG_{\text{uni}}^F$ are known by {\bf (A2)}, {\bf (A3)} in 
Section~\ref{typ4}. Finally, let $s=h_i$ where $i\not\in \{0,3,15\}$. Then, 
as in (\ref{abs33}), $\bL:=C_\bG(s)\subsetneqq \bG$ is a regular subgroup. 
In that case, Lusztig has shown that $\rho=\pm R_{\bL\subseteq \bP}^\bG
(\psi)$ for some $\psi\in \cE(\bL^F,s)$; see \cite[Theorem~3.3.22]{gema}. 
So, in order to determine $\rho(u)$ for $u\in \bG_{\text{uni}}^F$, we can 
use again the character formula (\ref{abs12})(a), combined with the 
knowledge of $Q_{\bL \subseteq \bP}^\bG$ (see Step~1) and the values 
$\psi(v)$ for $v\in\bL_{\text{uni}}^F$ (see {\bf (A4)}).
\end{abs}

\begin{abs} \label{step3} We show how Step 3 can be resolved. Assume 
that $\bL\subsetneqq \bG$ and let $\psi\in \Irr(\bL^F)$ be arbitrary. 
There is some semisimple $s\in \bL^F$ such that $\psi \in \cE(\bL^F,s)$. 
Let $\cE(\bG^F,s)=\{\rho_1,\ldots,\rho_r\}$. Then, by 
\cite[Prop.~3.3.20]{gema}, we have 
\begin{equation*}
R_{\bL\subseteq \bP}^\bG(\psi)=\sum_{i=1}^r m(\psi,\rho_i)\rho_i \qquad 
\mbox{where $m(\psi,\rho_i)\in \Z$ for $1\leq i\leq r$}.\tag{$*$}
\end{equation*}
If $s=1$ and $\psi\in \Unip(\bL^F)$, we can use Michel's
{\tt LusztigInductionTable}, as in (\ref{abs33}). Now assume that $s\neq 1$.
Then one could use the fact that $R_{\bL\subseteq \bP}^\bG$ commutes with
the Jordan decomposition of characters; see \cite[Theorem~4.7.2]{gema}.
But having the results of Steps~1 and~2 at our disposal, we can also argue
as follows. Let again $u_1,\ldots,u_N$ be representatives of the conjugacy 
classes of unipotent elements of~$\bG^F$. Using (\ref{abs12})(a), {\bf (A4)} 
and Step~1, we can compute the values:
\[ R_{\bL\subseteq \bP}^{\bG}(\psi)(u_k)=\sum_{v\in \bL_{\text{uni}}^F}
Q_{\bL\subseteq \bP}^{\bG}(u_k,v)\psi(v) \qquad\mbox{for $1\leq k 
\leq N$}.\]
Comparing with ($*$), we obtain equations
\[ \sum_{i=1}^r m(\psi,\rho_i)\rho_i(u_k)=R_{\bL\subseteq \bP}^{\bG}
(\psi)(u_k)=\mbox{known value} \qquad \mbox{for $1\leq k \leq N$}.\]
Using Step~2, we can check that the matrix $(\rho_i(u_k))_{1\leq i 
\leq r,1\leq k\leq N}$ has rank $r$, where $r\leq N$. (This would not
be true for $s=1$.) Hence, the above equations uniquely determine the 
numbers $m(\psi,\rho_i)$ for $1\leq i\leq r$. 
\end{abs}

\begin{abs} \label{step4} We show how Step 4 can be resolved. Let $\rho\in 
\Irr(\bG^F)$ and $g\in \bG^F$ be arbitrary. Let $i\in \{0,1,\ldots,76\}$ 
be such that $\rho\in \cE(\bG^F,h_i)$. If $i=0$, then $h_0=1$, $\rho$ is 
unipotent and we know the values of $\rho$ by Section~\ref{typ4}. Next,
let $i\in \{3,15\}$. Then, as already mentioned in (\ref{step2}), $\rho$ 
is uniform and so the values of $\rho$ are computable via {\bf (A2)}, 
{\bf (A3)}. Finally, let $i \not\in \{0,3,15\}$. Write $g=su=us$ where 
$s\in \bG^F$ is semisimple and $u\in \bG^F$ is unipotent. If $s=1$, then
the values $\rho(u)$ for $u\in \bG_{\text{uni}}^F$ are known by Step~2. 
Now let $s\neq 1$. If $C_\bG(s)$ has type $A_2\times A_2$, then $\rho(su)$ 
is already known by~(\ref{abs34}). Otherwise, we are in the situation 
of (\ref{abs33}) where $\bL:=C_\bG(s)\subsetneqq \bG$ is a regular 
subgroup. Let $\psi\in \Irr(\bL^F)$ and $(\bT,\theta)\in \fX(\bL,F)$ be
such that $\langle R_{\bT,\theta}^\bL,\psi\rangle_{\bL^F} \neq 0$; then, 
by \cite[Prop.~2.2.20]{gema}, we have $\psi(su)=\theta(s) \psi(u)$. So 
Schewe's formula, together with {\bf (A4)} and the result of Step~3, 
yield the value $\rho(su)$. 
\end{abs}

%In principle, we can 
%use the same strategie as in the previous section, 
%For $s\in \{h_3,h_{15}\}$, the discussion in (\ref{abs34}) did already 
%apply to all $\rho\in \Irr(\bG^F)$. So let us now fix $s\in \{h_0,h_1,
%\ldots,h_{76}\}\setminus \{h_3,h_{15}\}$; then $\bL:=C_\bG(s)$ is a
%regular subgroup; see again (\ref{abs33}). 

\smallskip
\noindent {\bf Acknowledgements}. I thank Jonas Hetz and Gunter Malle 
for comments on an earlier version. This article is a contribution to SFB-TRR 
195 by the DFG (Deutsche Forschungsgemeinschaft), Project-ID 286237555. 

%%%%%%%%%%%%%%%%%%%%%%%%%%%%%%%%%%%%%%%%%%%%%%%%%%%%%%%%%%%%%%%%%%%%%%%%%%%

\end{document}